\documentclass[12pt]{amsart}
\usepackage{graphicx} 
\usepackage{amsmath}
\usepackage{amsthm}
\usepackage{amsfonts}
\usepackage{amssymb}
\usepackage{enumitem}
\usepackage{mathtools}

\usepackage[style=alphabetic,sorting=nyt,backend=bibtex8, maxbibnames=99]{biblatex}
\bibliography{biblio.bib}

\usepackage{hyperref}
\hypersetup{colorlinks=true, urlcolor=black}

\usepackage[capitalize]{cleveref}

\makeatletter
\DeclareFontFamily{OMX}{MnSymbolE}{}
\DeclareSymbolFont{MnLargeSymbols}{OMX}{MnSymbolE}{m}{n}
\SetSymbolFont{MnLargeSymbols}{bold}{OMX}{MnSymbolE}{b}{n}
\DeclareFontShape{OMX}{MnSymbolE}{m}{n}{
    <-6>  MnSymbolE5
   <6-7>  MnSymbolE6
   <7-8>  MnSymbolE7
   <8-9>  MnSymbolE8
   <9-10> MnSymbolE9
  <10-12> MnSymbolE10
  <12->   MnSymbolE12
}{}
\DeclareFontShape{OMX}{MnSymbolE}{b}{n}{
    <-6>  MnSymbolE-Bold5
   <6-7>  MnSymbolE-Bold6
   <7-8>  MnSymbolE-Bold7
   <8-9>  MnSymbolE-Bold8
   <9-10> MnSymbolE-Bold9
  <10-12> MnSymbolE-Bold10
  <12->   MnSymbolE-Bold12
}{}

\let\llangle\@undefined
\let\rrangle\@undefined
\DeclareMathDelimiter{\llangle}{\mathopen}%
                     {MnLargeSymbols}{'164}{MnLargeSymbols}{'164}
\DeclareMathDelimiter{\rrangle}{\mathclose}%
                     {MnLargeSymbols}{'171}{MnLargeSymbols}{'171}
\makeatother

\usepackage[mode=multiuser,status=final,lang=english]{fixme} 
\fxsetup{theme=colorsig}

\FXRegisterAuthor{m}{am}{M}
\FXRegisterAuthor{r}{ar}{R}
\FXRegisterAuthor{l}{al}{L}

\newtheorem{theorem}{Theorem}[section]
\newtheorem{maintheorem}{Theorem}

\newtheorem{proposition}[theorem]{Proposition}
\newtheorem{corollary}[theorem]{Corollary}
\newtheorem{lemma}[theorem]{Lemma}
\newtheorem{claim}{Claim}[theorem]
\theoremstyle{definition}

\newtheorem{question}[theorem]{Question}
\newtheorem{definition}[theorem]{Definition}

\theoremstyle{remark}
\newtheorem{remark}[theorem]{Remark}

\crefname{maintheorem}{Theorem}{Theorems}

\begin{document}

\title[Some questions on entangled linear orders]{Some questions on\\entangled linear orders}

\author[Carroy]{Rapha\"{e}l Carroy}
\address{Università degli Studi di Torino, Dipartimento di Matematica ``G. Peano", Via Carlo Alberto 10, 10123 Torino, Italy}
\curraddr{}
\email{raphael.carroy@unito.it}
\author[Levine]{Maxwell Levine}
\address{Albert-Ludwigs-Universit{\"{a}}t Freiburg, Mathematisches Institut, Abteilung f\"{u}r math. Logik, Ernst–Zermelo–Stra{\ss}e 1, 79104 Freiburg im Breisgau, Germany}
\email{maxwell.levine@mathematik.uni-freiburg.de}
\author[Notaro]{Lorenzo Notaro}
\address{University of Vienna, Institute of Mathematics, Kurt G\"{o}del Research Center, Kolingasse 14-16, 1090 Vienna, Austria}
\email{lorenzo.notaro@univie.ac.at}

\thanks{The first and third authors were partially supported by the Italian PRIN 2022 ``Models, sets and classifications” prot. 2022TECZJA and ``Logical Methods in Combinatorics" prot. 2022BXH4R5.
The research of the third author was funded in whole or in part by the Austrian Science Fund (FWF) \href{https://www.fwf.ac.at/en/research-radar/10.55776/ESP1829225}{10.55776/ESP1829225}. For open access purposes, the author has applied a CC BY public copyright license to any author accepted manuscript version arising from this submission.
The authors are grateful to the anonymous referee for their useful suggestions and remarks.}

\subjclass[2020]{Primary 03E05, Secondary 03E50, 03E15, 06A05}
\keywords{entangled sets, entangled linear orders, Baumgartner's axiom, diagonalization, Sierpinski}

\begin{abstract}
Entangled linear orders were introduced by Abraham and Shelah in \cite{avraham1981martin}. Todor{\v c}evi{\'c} \cite{MR799818} showed that these linear orders exist under $\mathsf{CH}$. We prove the following results: \vspace{0.4em}
\begin{enumerate}
\itemsep0.3em
\item If $\mathsf{CH}$ holds, then, for every $n > 1$, there is an $n$-entangled linear order which is not $(n+1)$-entangled.
\item If $\mathsf{CH}$ holds, then there are two homeomorphic sets of reals $A,B \subseteq \mathbb{R}$ such that $A$ is entangled but $B$ is not $2$-entangled.
\item If $\mathbb{R} \subseteq \mathsf{L}$, then there is an entangled $\Pi_1^1$ set of reals.
\item If $\diamondsuit$ holds, then there is a $2$-entangled non-separable linear order.
\end{enumerate}

\end{abstract}

\maketitle


\section{Introduction}

Cantor proved that all countable dense unbounded linear orders are isomorphic. The countability hypothesis in Cantor's result cannot be avoided in general. However, in the early 1970s, Baumgartner proved that consistently all $\aleph_1$-dense separable linear orders are isomorphic \cite{MR317934}. This statement is known as \emph{Baumgartner's axiom}. In their work \cite{avraham1981martin}, Abraham and Shelah introduced entangled linear orders as combinatorial objects witnessing the failure of Baumgartner's axiom.

An uncountable linear order $L$ is said to be \emph{$n$-entangled}, for some positive integer $n > 0$,
if for every uncountable collection $F$ of increasing, pairwise disjoint $n$-tuples of elements of $L$ and for every $t \in {}^n 2$,
there are $x, y \in F$ such that, for every $i < n$, $x_i < y_i$ if and only if $t_i = 0$.
A linear order is called \emph{entangled} if it is $n$-entangled for every $n > 0$ (see \cref{def:entangled} for a more general definition).

Entangled linear orders satisfy a strong form of rigidity\footnote{Indeed, entangled linear orders are called \emph{hyper-rigid} in \cite{MR776283}.}: they have no nontrivial self-embeddings\footnote{A linear order with this property is sometimes called \emph{exact} in the literature (see e.g. \cite{MR70683, MR95131}).} (a linear order with this property was first constructed in \cite{MR1919}); actually, they do not even have nontrivial monotone self-maps; furthermore, every two disjoint uncountable subsets of an entangled linear order have different order-types (a linear order with this property was first constructed, under $\mathsf{CH}$, in \cite{MR4862}).

Todor{\v c}evi{\'c} \cite{MR799818} showed that these linear orders provide a powerful method to construct counterexamples relative to various problems,
such as the productivity of chain conditions and the square bracket partition relation. He also observed that every $2$-entangled linear order is ccc and that every $3$-entangled linear order is separable (see \cref{sec:prel}). In particular, every entangled linear order is isomorphic to a suborder of the real line $\mathbb{R}$. Therefore, when studying entangled linear orders, it is not restrictive to only consider sets of reals that are entangled with respect to their natural linear ordering. For recent applications of entangled linear orders see e.g. \cite{MR3385104, MR4128471, MR4819970, MR4646731}.

The existence of an entangled linear order is independent of $\mathsf{ZFC}$: on one hand, these linear orders exist under $\mathsf{CH}$ (actually, under $\mathrm{cf}(2^{\aleph_0}) = \omega_1$) \cite{MR799818, MR776283}; on the other hand, they do not exist under Baumgartner's axiom \cite{avraham1981martin}.

In this article, we address  the following natural questions:\vspace{0.3em}

\begin{enumerate}
\itemsep0.3em
\item Does an $n$-entangled linear order need to be $(n+1)$-entangled?
\item Is entangledness a topological property? That is to say, if we have two homeomorphic sets of reals $A, B \subseteq \mathbb{R}$ with $A$ entangled, does it follow that $B$ is also entangled?
\item How definable can an entangled set of reals be?
\item Under which conditions is there a $2$-entangled linear order which is not separable? \vspace{0.3em}
\end{enumerate}

Our main results are the following. \lnote*{}{The first two are instances of more general statements proved later, formulated with reference to cardinal characteristics of the continuum (see Theorems~\ref{thm:1s} and~\ref{thm:2s}).}

\begin{maintheorem}\label{thm:1}
Assume $\mathsf{CH}$. Then, for every $n > 1$, there exists a $n$-entangled set of reals which is not $(n+1)$-entangled.
\end{maintheorem}

\begin{maintheorem}\label{thm:2}
Assume $\mathsf{CH}$. Then, there are two homeomorphic sets of reals $A,B\subseteq \mathbb{R}$ such that $A$ is entangled but $B$ is not $2$-entangled.
\end{maintheorem}

Recall that $\Pi^0_1$ sets are complements of semirecursive sets, also called effectively open sets. 
Projections of $\Pi^0_1$ sets are denoted $\Sigma^1_1$, and their complements are the $\Pi^1_1$ sets. As we argue in \cref{sec:thm3}, every analytic set of reals is not entangled. In this sense, our next result is optimal.

\begin{maintheorem}\label{thm:3}
Suppose that $ \mathbb{R}\subseteq \mathsf{L}$. Then there exists a $\Pi_1^1$ set of reals which is entangled.
\end{maintheorem}

Since, as we already noted, every $2$-entangled linear order is ccc, a $2$-entangled non-separable linear order is, in particular, a Suslin line. In particular, we conclude that $\mathsf{CH}$---from which the existence of an entangled linear order follows---does not suffice to prove the existence of a $2$-entangled non-separable linear order. Indeed, it is well-known that Suslin's Hypothesis is consistent with $\mathsf{CH}$ \cite{zbMATH03453594}. However, we prove that the existence of a $2$-entangled non-separable linear order follows from $\diamondsuit$. Note that the existence of such a linear order is already known to be consistent by a forcing argument due to Krueger \cite{MR4080091}.

\begin{maintheorem}\label{thm:4}
Suppose that $\diamondsuit$ holds. Then, there exists a $2$-entangled non-separable linear order.
\end{maintheorem}

The underlying general technique used to prove all four results is the diagonalization method. More precisely, the proofs of the first two theorems use a diagonalization argument \emph{à la} Sierpi\'{n}ski (e.g. \cite{zbMATH02598169, zbMATH03124831}). The proof of our third result uses a diagonalization argument similar to the construction (originally due to G\"{o}del \cite[p. 170]{MR2731169}) of a $\Pi_1^1$ set without the perfect set property. Finally, the diagonalization construction of our fourth result is similar to the construction due to Jensen of a Suslin line under $\diamondsuit$ \cite{MR756630, Devlin:1984aa}.

In \cref{sec:prel} we provide a more general definition of entangled linear order and we give a brief account of what is known on this topic.
\cref{sec:thm1,sec:thm2,sec:thm3,sec:thm4} are then devoted to the proofs of \cref{thm:1,thm:2,thm:3,thm:4}. We provide some open questions along the way.

\subsection*{Notation} For most of our paper the notation is standard---see e.g. \cite{MR1321597, MR1940513}. Given two tuples $s$ and $t$, we denote the concatenation of $s$ with $t$ by $s^\smallfrown t$. In Section~\ref{sec:thm3}, following set-theoretic \lnote*{}{practice}, we refer to elements of the Baire space ${}^\omega \omega$ as ``reals", and we employ some notation from \cite{MR2731169}.

\section{Preliminaries}\label{sec:prel} 

There exists a more general notion of entangled linear order than the one we gave in the introduction, one that extends to 
cardinalities larger than the continuum. This more general definition (\cref{def:entangled}) will be useful in our discussion.

We say that two tuples $x,y$ are \emph{disjoint} if they have no element in common---that is if $\mathrm{ran}(x) \cap \mathrm{ran}(y) = \emptyset$. 

Given a linear order $(L, \le)$ and two $n$-tuples $e$ and $e'$ of elements of $L$, we let $t(e,e') \in {}^n 2$ be the binary $n$-tuple defined as follows: \lnote*{}{for every $i < n$, the $i$-th element of $t(e,e')$ is $1$ if and only if $e'_i < e_i$. }

For each $t \in {}^{<\omega} 2$, we denote by $\neg t$ the tuple $\langle 1-t_i \mid i < \text{length}(t)\rangle$. Clearly, if the $n$-tuples $e$ and $e'$ are disjoint, then $t(e,e') = \neg t(e', e)$.  We say that $e$ and $e'$ \emph{realize} $s$ if either $t(e,e') = s$ or $t(e,e') = \neg s$.

Given a set $S$ of mutually disjoint $n$-tuples of elements of some linear order, and given a binary $n$-tuple $t$, we say that $S$ \emph{avoids} $t$ if no couple of distinct elements of $S$ realizes $t$. 
For each $n > 0$, we call the $n$-tuple $\langle 0\rangle^n =  \langle 0, 0, \dots, 0\rangle \in {}^n 2$ \emph{the increasing type}.

We point out a trivial yet useful observation.
The usual order $\leq$ on real numbers is closed but not open as a subset of $\mathbb{R}^2$ and it becomes clopen\footnote{That is, both open and closed.} on $\mathbb{R}^2$ minus the diagonal.
Therefore, for every $t \in {}^n 2$, the following (symmetric) binary relation $R_t$ on $\mathbb{R}^{2n}$ is open: $e \mathrel{R_t} e'$ if $e$ and $e'$ are disjoint and realize $t$.

\begin{definition}[\cite{avraham1981martin, MR799818}]\label{def:entangled} Given a linear order $(L, \le)$, an infinite cardinal $\kappa$ and a positive integer $n > 0$, we say that $L$ is \emph{$(\kappa, n)$-entangled} if $|L| \ge \kappa$ \lnote*{}{and} for every collection $F$ of size $\kappa$ of 1-1 mutually disjoint $n$-tuples of $L$ and for every binary $n$-tuple $t$, there are two distinct elements of $F$ that realize $t$. A linear order which is $(\kappa, n)$-entangled for every $n > 0$ is said to be \emph{$\kappa$-entangled}. If $\kappa$ is not mentioned, it is assumed that $\kappa = \aleph_1$. 
\end{definition}

We now state some known theorems regarding entangled linear orders. The first one is about some topological consequences of entangledness. In particular, it implies that every $3$-entangled linear order is isomorphic to a set of reals. This theorem is stated without proof in \cite{MR799818}. For a proof, we refer the reader to \cite{MR4080091}.

\begin{theorem}[\cite{MR799818}]
Every $2$-entangled linear order is ccc, and every $3$-entangled linear order is separable.
\end{theorem}
\begin{corollary}
Every $3$-entangled linear order is isomorphic to a suborder of $\mathbb{R}$.
\end{corollary}
\begin{proof}
Every $2$-entangled linear order has countably many gaps\footnote{Given a linear order $L$, a \emph{gap} of $L$ is a pair $(a,b) \in L$ with $a < b$ such that there is no $c \in L$ with $a < c < b$.}: otherwise, letting $\{ (a_\xi, b_\xi) \mid \xi < \omega_1\}$ be an uncountable family of gaps, the map $a_\xi \mapsto b_\xi$ would be uncountable, 1-1 and increasing, against $2$-entangledness of the underlying linear order. 

Therefore, it follows from the classical Dedekind completion argument that every $3$-entangled linear order, being separable and having countably many gaps, is isomorphic to a suborder of $\mathbb{R}$.
\end{proof}

The next theorem is an example of the numerous applications of entangled linear orders as studied by Shelah, Todor{\v c}evi{\'c} and others (see e.g. \cite{MR799818, MR776283, MR1467577, MR1812182, chapital2025nentangledsetn1entangledsets}).

\begin{theorem}[\cite{MR799818}]
If there is an entangled linear order, then there are two ccc partial
orders whose product is not ccc.
\end{theorem}

The next results are mainly concerned with the existence or nonexistence of entangled linear orders:

%
%
%
%
%
%

\begin{itemize}
\itemsep0.4em
\item There exists a $\mathrm{cf}(2^{\aleph_0})$-entangled linear order \cite{MR799818, MR776283}.
\item Adding a single Cohen real or random real adds an entangled linear order \cite[p. 55]{MR980949}.
\item It follows from the Open Graph Axiom that there are no $2$-entangled sets of reals \cite[Proposition 8.4]{MR980949}.
\item It follows from $\mathsf{MA}(\aleph_1)$ that there are no entangled linear orders \cite{avraham1981martin}.
\item It is consistent, relative \lnote*{}{to} $\mathsf{ZF}$, that $\mathsf{MA}(\aleph_1)$ holds and that, for every $n > 0$, there is an $n$-entangled set of reals \cite{MR801036}.
\end{itemize}

Next, we note that it directly follows from the Erd\H{o}s-Dushnik-Miller theorem \cite{MR4862} (see also \cite[Theorem 9.7]{MR1940513} for a proof) that the definition of $(\kappa, n)$-entangled linear order is equivalent to the seemingly stronger version in which we do not only require the existence of just two distinct elements of $F$ to realize the given type $t$, but we ask for the existence of an \emph{infinite} subset of $F$ whose elements, mutually, realize $t$:
\begin{theorem}[Erd\H{o}s-Dushnik-Miller]
Any infinite graph $G = (V, E)$ not containing an independent set of size $|V|$ contains an infinite complete subgraph.
\end{theorem}
\begin{proposition}
Let $L$ be a $(\kappa, n)$-entangled linear order. Then, for every collection $F$ of size $\kappa$ of 1-1 mutually disjoint $n$-tuples of $L$ and for every binary $n$-tuple $t$, there exists an infinite $H\subseteq F$ such that for all distinct $e,e' \in H$, $e$ and $e'$ realize $t$.
\end{proposition}
\begin{proof}
Let us fix a $(\kappa, n)$-entangled linear order $L$, a collection $F$ of size $\kappa$ of 1-1 mutually disjoint $n$-tuples of $L$ and a binary $n$-tuple $t$.

Consider the binary relation $R_t$ on $F$ defined by: $e \mathrel{R_t} e'$ if $e$ and $e'$ realize $t$. By entangledness, the graph $(F, R_t)$ does not contain an independent set of size $|F| = \kappa$. Therefore, by the Erd\H{o}s-Dushnik-Miller Theorem, the graph $(F, R_t)$ contains an infinite complete subgraph. In other words, there exists an infinite $H\subseteq F$ such that for all distinct $e,e' \in H$, $e$ and $e'$ realize $t$.
\end{proof}

Finally, let us provide a proof of the following proposition from \cite[\S 1]{MR799818} that provides a criterion to show that a set $E\subseteq \mathbb{R}$ is $(2^{\aleph_0}, n)$-entangled.

\begin{proposition}\label{prop:criterionforentangledness}
Let $E\subseteq \mathbb{R}$ be a set of cardinality $2^{\aleph_0}$. Assume that there exists an enumeration $\langle r_\alpha \mid \alpha < 2^{\aleph_0}\rangle$ of $E$ such that, for every continuous function $f$ from a $G_\delta$ subset of $\mathbb{R}^{n-1}$ to $\mathbb{R}$,
there exists an $\alpha < 2^{\aleph_0}$ with
\begin{equation}\label{eq:criterionforentangledness}
f\big[\{r_\gamma \mid \gamma < \beta\}^{n-1}\big] \cap E \subseteq \{r_\gamma \mid \gamma < \beta\},
\end{equation}
for all $\beta \ge \alpha$. Then $E$ is $(2^{\aleph_0},n)$-entangled.
\end{proposition}

\begin{lemma}\label{lemma:finalarg}
Let $E$ be a $(2^{\aleph_0}, n)$-entangled set of reals, for some $n > 0$, and let $g:F \rightarrow E$ be a function with $F$ being a subset of $E^n$ of cardinality $2^{\aleph_0}$ whose elements are 1-1 and pairwise disjoint. If $g$ has $2^{\aleph_0}$-many discontinuity points, then any $t \in {}^{n+1}2$ is realized by two distinct $x,y \in \mathrm{graph}(g)$.
\end{lemma}
\begin{proof}
Let us assume without loss of generality that $E\subseteq [0,1]$, and let $G \subseteq [0,1]^{n+1}$ be the closure of the graph of the map $g$.

It follows from the compactness of the unit interval that for every $u \in F$, $u$ is a continuity point of $g$ if and only if $|G_u| = 1$ or, equivalently, $G_u = \{g(u)\}$, where $G_u \coloneqq \{y \in [0,1] \mid u^\smallfrown y \in G\}$. We conclude from our hypothesis that there are $2^{\aleph_0}$-many $u \in F$ with $|G_u| > 1$. Therefore, there must be a rational $q \in \mathbb{Q}$ and a subset $D\subseteq F$ of cardinality $2^{\aleph_0}$ such that, for every $u \in D$, there are $a,b \in G_u$ with $a < q < b$. 

Fix some $t \in {}^{n+1} 2$ and assume without loss of generality that $t(n) = 0$. By \lnote*{}{entangledness}, there are $u,v \in D$ such that $t(u, v) = t \upharpoonright n$. Then, by definition of $D$, we can pick $x,y \in F$ sufficiently close to $u$ and $v$, respectively, such that $g(x) < q < g(y)$ and $t(x, y) = t(u, v) = t \upharpoonright n$. Therefore, $x^\smallfrown g(x)$ and $y^\smallfrown g(y)$ realize $t$.
\end{proof}

\begin{proof}[Proof of \cref{prop:criterionforentangledness}]
We prove the proposition by induction on $n > 0$. The base case $n = 1$  follows from the fact that $E$ has cardinality $2^{\aleph_0}$. Suppose that $n > 1$ and fix an enumeration $\langle r_\alpha \mid \alpha < 2^{\aleph_0}\rangle$ of $E$ that satisfies the hypothesis of the proposition. By the induction hypothesis, $E$ is $(2^{\aleph_0}, n-1)$-entangled.

Fix a collection $F$ of $2^{\aleph_0}$-many 1-1 pairwise disjoint $n$-tuples of elements of $E$.
By passing to a subset of $F$, we can assume without loss of generality that all the tuples of $F$ are increasing with respect to the fixed enumeration of $E$.

Suppose, towards a contradiction, that $F$, seen as a function from a subset of $E^{n-1}$ to $E$, has less than $2^{\aleph_0}$-many discontinuity points. Then, passing to a subset if necessary, we can suppose that $F$ is continuous. By Kuratowski's extension theorem \cite[Theorem 3.8]{MR1321597}, there exists a continuous map $\tilde{f}$ from a $G_\delta$ subset of $\mathbb{R}^{n-1}$ to $\mathbb{R}$ that extends $F$. Since $F$ has cardinality $2^{\aleph_0}$, and its elements are increasing with respect to the fixed enumeration of $E$, we conclude that for cofinally many $\beta < 2^{\aleph_0}$, the set $\tilde{f}[\{r_\gamma \mid \gamma < \beta\}^{n-1}] \cap E $ is not included in $\{r_\gamma \mid \gamma < \beta\}$, against \eqref{eq:criterionforentangledness}. Hence the contradiction.

Therefore, $F$ has $2^{\aleph_0}$-discontinuity points. But then, by the $(2^{\aleph_0}, n-1)$-entangledness of $E$ and \cref{lemma:finalarg}, we conclude that every $t \in {}^{n} 2$ is realized by a pair of elements of $F$. Thus, $E$ is $(2^{\aleph_0}, n)$-entangled.
\end{proof}

\section{Separating $n$-entangledness}\label{sec:thm1}

\cref{thm:1} is a direct consequence of the following theorem. 
Recall that $\mathrm{cov}(\mathcal{B})$ is the least cardinal $\kappa$ such that a perfect Polish space (it does not matter which) can be expressed as the union of $\kappa$-many meager sets. 

\begin{theorem}\label{thm:1s}
Assume $\mathrm{cov}(\mathcal{B}) = 2^{\aleph_0}$. Then, for every $n > 1$, there exists a $(2^{\aleph_0}, n)$-entangled set of reals which is not $(2^{\aleph_0}, n+1)$-entangled.
\end{theorem}
\begin{lemma}\label{lemma:1}
For every $n > 1$, there exists a countable set $D \subseteq \mathbb{R}^{n+1}$ of 1-1 mutually disjoint $(n+1)$-tuples of reals such that:
\begin{enumerate}[label={\upshape (\arabic*)}]
\itemsep0.3em
\item $D$ avoids the increasing type.
\item For every nonempty open $U \subseteq \mathbb{R}^n$, for every $i< n$, there are two disjoint $e,e' \in D$ such that:\vspace{0.3em}
\begin{enumerate}[label={\upshape (2\alph*)}]
\itemsep0.3em
\item $t(e,e') = \langle 0 \rangle^i{}^\smallfrown \langle 1 \rangle{}^\smallfrown \langle 0 \rangle^{n-i}$.
\item $e \upharpoonright n$ and $e' \upharpoonright n$ belong to $U$.
\item There are two disjoint $w,w' \in U$ that realize the increasing type and such that $w_j = e_j$ and $w'_j = e'_j$ for all $j < n$ with $j \neq i$.
\end{enumerate}
\end{enumerate}
\end{lemma}
\begin{proof}
We define by induction a sequence $\langle (r_k^0, r_k^1) \mid k \in\omega\rangle$ of couples of $(n+1)$-tuples of reals such that the set $D = \{r_k^i \mid k\in\omega \text{ and } i \in 2\}$ satisfies the desired properties.

Fix an enumeration $\langle U_k \mid k\in\omega\rangle$ of a countable basis of $\mathbb{R}^n$. Fix also a surjective map $k \mapsto \langle (k)_0, (k)_1\rangle$ from $\omega$ to $\omega\times n$.

Now we deal with the inductive construction of the sequence. Suppose we have defined $(r_m^0, r_m^1)$ for every $m < k$, towards defining $(r_k^0, r_k^1)$, with the induction hypothesis being that $D_k \coloneqq \{r_m^i \mid m < k \text{ and } i \in 2\}$ is a set of 1-1 mutually disjoint $(n+1)$-tuples that avoids the increasing type. 

For each $n$-tuple $s$ of reals, let
\begin{align*}
u_k(s) &\coloneqq \min \{x \in \mathbb{R} \mid \exists w \in D_k \ (t(s, w \upharpoonright n) = \langle 1 \rangle^n \text{ and } w_n = x)\}\\
l_k(s) &\coloneqq \max \{x \in \mathbb{R} \mid \exists w \in D_k \ (t(s, w \upharpoonright n) = \langle 0 \rangle^n \text{ and } w_n = x)\}
\end{align*}
where we use the convention that $\min \emptyset = +\infty$ and $\max \emptyset = -\infty$.
\begin{claim}\label{claim:lemma:1}
$l_k(s) < u_k(s)$ for every $s$.
\end{claim}
\begin{proof}
If there is no $w \in D_k$ such that $t(s, w \upharpoonright n) = \langle 1\rangle^n$ the claim trivially follows, and similarly if there is no $w' \in D_k$ such that $t(s, w' \upharpoonright n) = \langle 0\rangle^n$. So we can assume that such $w$ and $w'$ exist and fix them. Clearly, $w \neq w'$ and $t(w \upharpoonright n, w' \upharpoonright n) = \langle 0\rangle^n$. By the induction hypothesis, $w$ and $w'$ do not realize the increasing type. Thus, we must have $w'_n < w_n$. It directly follows from the definition of $u$ and $l$ that $l_k(s) < u_k(s)$.
\end{proof}

Since $D_k$ is finite, we can pick two 1-1 disjoint tuples $p, q \in U_{(k)_0}$ such that:\vspace{0.3em}
\begin{enumerate}[label=(\alph*)]
\itemsep0.3em
\item $(\mathrm{ran}(p) \cup \mathrm{ran}(q)) \cap \bigcup_{w \in D_k} \mathrm{ran}(w) = \emptyset$.
\item $t(p,q) = \langle 0 \rangle^{(k)_1}{}^\smallfrown \langle 1 \rangle{}^\smallfrown \langle 0 \rangle^{n-1-(k)_1}$.
\item $l_k(p) = l_k(q)$ and $u_k(p) = u_k(q)$.
\item Letting $\tilde{p},\tilde{q} \in \mathbb{R}^n$ be such that $\tilde{p}_j = p_j$ and $\tilde{q}_j = q_j$ for all $j \neq (k)_1$ and $\tilde{p}_{(k)_1} = q_{(k)_1}$ and $\tilde{q}_{(k)_1} = p_{(k)_1}$, we have $\tilde{p},\tilde{q} \in U_{(k)_0}$.\vspace{0.3em}
\end{enumerate}
Indeed, by the finitude of $D_k$, we can always find a nonempty open set $V\subseteq U_{(k)_0}$ such that $\mathrm{ran}(s) \cap \bigcup_{w \in D_k} \mathrm{ran}(w) = \emptyset$ for all $s \in V$; furthermore, to satisfy (c) and (d) we just need to choose, in $U_{(k)_0}$, $p$ and $q$ so that they are close enough (e.g. with respect to the Euclidean metric on $\mathbb{R}^n$); thus, any two 1-1, disjoint, close enough $p$ and $q$ realizing $\langle 0 \rangle^{(k)_1}{}^\smallfrown \langle 1 \rangle{}^\smallfrown \langle 0 \rangle^{n-1-(k)_1}$ taken inside $V$ will satisfy (a)-(d).

Again by the finitude of $D_k$, we can pick $x,y \in (l_k(p), u_k(p))$ with $x < y$ such that $x,y \not\in \bigcup_{w \in D_k} \mathrm{ran}(w)\cup \mathrm{ran}(p) \cup \mathrm{ran}(q)$. Set $r_k^0 = p^\smallfrown \langle x\rangle$ and $r_k^1 = q^\smallfrown \langle y\rangle$. Clearly, $r_k^0$ and $r_k^1$ realize $\langle 0 \rangle^{(k)_1}{}^\smallfrown \langle 1 \rangle{}^\smallfrown \langle 0 \rangle^{n-(k)_1}$.

It follows from the choices of $p,q,x,y$ that $D_{k+1} = D_k \cup \{r_k^0, r_k^1\}$ is still a set of 1-1 mutually disjoint $(n+1)$-tuples that avoids the increasing type,
and (d) guarantees that $\tilde{p}$ and $\tilde{q}$ are in $U_{(k)_0}$ and realize the increasing type, thus ensuring $(2c)$.
This finishes the inductive construction. It is straightforward to check that the set $D$ as defined satisfies the desired properties.
\end{proof}

\begin{proof}[Proof of \cref{thm:1s}]
Fix a set $D$ of $(n+1)$-tuples that satisfies the statement of \cref{lemma:1}. Fix also an enumeration $\langle f_\alpha \mid \alpha < 2^{\aleph_0}\rangle$ of the continuous functions from $G_\delta$ subsets of $\mathbb{R}^{n-1}$ to $\mathbb{R}$. We can assume that $f_0(s) = s_0$ for every $s \in  \mathbb{R}^{n-1}$.

We define inductively a sequence $\langle r_\alpha \mid \alpha < 2^{\aleph_0} \rangle$ of 1-1 mutually disjoint $(n+1)$-tuples of reals that avoids the increasing type, but such that the set $ E \coloneqq \bigcup_{\alpha < 2^{\aleph_0}} \mathrm{ran}(r_\alpha)$ is $(2^{\aleph_0}, n)$-entangled. In particular, $E$ is a $(2^{\aleph_0}, n)$-entangled set of reals which is not $(2^{\aleph_0}, n+1)$-entangled.

Suppose we have defined $r_\beta$ for every $\beta < \alpha$ towards defining $r_\alpha$, with the induction hypothesis being that the set
\[
P_\alpha \coloneqq D \cup \{r_\beta \mid \beta < \alpha\}
\]
is a set of 1-1 mutually disjoint $(n+1)$-tuples that avoids the increasing type (for the base case $\alpha = 0$, this is guaranteed by \cref{lemma:1}). Let
\[
Q_\alpha \coloneqq \bigcup_{w \in P_{\alpha}} \mathrm{ran}(w),
\]
and
\begin{multline*}
A_\alpha \coloneqq \big\{s \in \mathbb{R}^n \mid \exists i < n \ \exists \beta \le \alpha \ \exists u \in (Q_\alpha \cup \mathrm{ran}(s \upharpoonright i))^{n-1} \\ \text{ such that }s_i = f_\beta(u)\big\}.
\end{multline*}
Furthermore, for each $s \in \mathbb{R}^n$ let 
\begin{align*}
u_\alpha(s) &\coloneqq \inf \{x \in \mathbb{R} \mid \exists w \in P_\alpha \ (t(s, w \upharpoonright n) = \langle 1 \rangle^n \text{ and } w_n = x)\},\\[0.4em]
l_\alpha(s) &\coloneqq \sup \{x \in \mathbb{R} \mid \exists w \in P_\alpha \ (t(s, w \upharpoonright n) = \langle 0 \rangle^n \text{ and } w_n = x)\}.
\end{align*}
By an argument analogous to the one used in \cref{claim:lemma:1}, we can show that $l_\alpha(s) \le u_\alpha(s)$ for every $s \in \mathbb{R}^n$.  Next, we prove the following two claims, the second of which is key.
\begin{claim}\label{claim:thm:1-1}
$A_\alpha$ is the union of ${<}2^{\aleph_0}$-many nowhere dense sets.
\end{claim}
\begin{proof}
Consider the following map:
\begin{align*}
g:\mathbb{R}^{n} \times (Q_\alpha \cup n)^{n-1} &\longrightarrow \mathbb{R}^{n-1}\\
(s, c) &\longmapsto w\text{, where } w_i \coloneqq \begin{cases}
c_i &\text{ if } c_i \in Q_\alpha,\\
s_j &\text{ if } c_i=j < n.
\end{cases}
\end{align*}
Clearly, for any fixed $c \in (Q_\alpha \cup n)^{n-1}$, the map $g({\cdot}, c) : \mathbb{R}^n \rightarrow \mathbb{R}^{n-1}$ is continuous. It suffices to show that given any fixed $i <n, \beta < \alpha$ and $c \in (Q_\alpha \cup i)^{n-1}$, the set of all those $s \in \mathbb{R}^n$ such that $s_i = f_\beta \circ g(s, c)$ is nowhere dense. But this is a direct consequence of the continuity of $f_\beta$ and of $g({\cdot}, c)$. 
\end{proof}

\begin{claim}\label{claim:thm:1}
There exists an $s \in \mathbb{R}^n \setminus A_\alpha$ and an $x \in \mathbb{R}$ such that
\begin{enumerate}[label={\upshape (\arabic*)}]
\itemsep0.3em
\item $l_\alpha(s) \le x \le u_\alpha(s)$,
\item For all $u \in (Q_\alpha \cup \mathrm{ran}(s))^{n-1}$, for all $\beta \le \alpha$, $f_\beta(u) \neq x$.
\end{enumerate}
\end{claim}
\begin{proof}
If there were an $s \in \mathbb{R}^n \setminus A_\alpha$ such that $l_\alpha(s) < u_\alpha(s)$ we would be done. Indeed, since there are only ${<}2^{\aleph_0}$-many possible $u \in (Q_\alpha \cup \mathrm{ran}(s))^{n-1}$ and $\beta \le \alpha$, there would certainly be an $x \in (l_\alpha(s), u_\alpha(s))$ satisfying (2). Thus, we need to deal with the case in which $l_\alpha(s) = u_\alpha(s)$ for every $s \in \mathbb{R}^n \setminus A_\alpha$.

Towards a contradiction, suppose that for every $s \in \mathbb{R}^n\setminus A_\alpha$ there exists a $u \in (Q_\alpha \cup \mathrm{ran}(s))^{n-1}$ and a $\beta \le \alpha$ such that $l_\alpha(s) = u_\alpha(s) = f_\beta(u)$. Consider the map $g$ defined in \cref{claim:thm:1-1}.
Then, another way of stating our (absurd) assumption is that for every  $s \in \mathbb{R}^n \setminus A_\alpha$ there exists a $c_s \in  (Q_\alpha \cup n)^{n-1}$ and a $\beta_s \le \alpha$ such that $l_\alpha(s) = u_\alpha(s) = f_{\beta_s} \circ g (s, c_s)$.

Since we are assuming $\mathrm{cov}(\mathcal{B}) = 2^{\aleph_0}$, and since, by \cref{claim:thm:1-1}, $A_\alpha$ is the union of ${<}2^{\aleph_0}$-many nowhere dense sets, it follows that there must exist a nonempty open set $U$, a $\beta \le \alpha$ and a $c \in (Q_\alpha \cup n)^{n-1}$ such that $U$ is included in the closure of the set 
\[
R \coloneqq \{s \in \mathbb{R}^n \setminus A_\alpha \mid (\beta_s, c_s) = (\beta, c)\}.
\]

Now, since $c$ has length $n-1$, there is an $i < n$ such that $i\not\in \mathrm{ran}(c)$. Fix one such $i$ and fix $e,e'$ in $D$ that satisfy (2a)-(2c) of \cref{lemma:1} for $U$ and $i$. Let $w,w'$ witness (2c) for $e, e'$.

For every $s, s' \in R$ such that $t(s, e \upharpoonright n) = \langle 1\rangle^n$ and $t(s',e'  \upharpoonright n) = \langle 0\rangle^n$, we have $u_\alpha(s) \le e_n$ and $l_\alpha(s') \ge e'_n$. Thus,
\[
f_{\beta} \circ g(s, c) = u_\alpha(s) \le e_n < e'_n \le l_\alpha(s') = f_{\beta} \circ g(s', c).
\]
Note that the value of $f_{\beta} \circ g(s, c)$ does not depend on the $i$-th element of $s$, since $i \not \in \mathrm{ran}(c)$. Therefore, by continuity of $f_\beta$ and by the density of $R$ in $U$, we can find $z,z' \in R \cap U$ close enough to $w$ and $w'$, respectively, such that
\begin{enumerate}[label=(\alph*)]
\itemsep0.3em
\item $t(z,z') = \langle 0 \rangle^n$,
\item $f_{\beta} \circ g(z, c) \le e_n$,
\item $f_{\beta} \circ g(z', c) \ge e_n'$.
\end{enumerate}

By density of $\{a \upharpoonright n \mid a \in D\}$ in $\mathbb{R}^n$, there must be $a,a' \in D$ (with $a \upharpoonright n$ and $a' \upharpoonright n$ close enough to $z$ and $z'$, respectively)  with $a\upharpoonright n$ and $a' \upharpoonright n$ in $U$ such that 
\begin{enumerate}[label=(\roman*)]
\itemsep0.3em
\item $t(a \upharpoonright n,a' \upharpoonright n) = \langle 0 \rangle^n$,
\item $t(z, a \upharpoonright n) = \langle 0\rangle^n$,
\item $t(z',a' \upharpoonright n) = \langle 1\rangle^n$.
\end{enumerate}
In particular, $a_n \le l_\alpha(z) = f_{\beta} \circ g(z, c)$ and  $a'_n \ge u_\alpha(z') = f_{\beta} \circ g(z', c)$. Since $f_{\beta} \circ g(z, c) < f_{\beta} \circ g(z', c)$, we conclude $a_n < a'_n$. But then $a$ and $a'$ realize the increasing type. In particular, we have found two distinct elements of $D$ that realize the increasing type, contradiction.
\end{proof}
Fix an $s \in \mathbb{R}^n \setminus A_\alpha$ and an $x$ satisfying the statement of our claim, and set $r_\alpha = s^\smallfrown \langle x\rangle$.

\lnote*{}{
It remains to prove that the inductive hypothesis is still satisfied, i.e. that $D \cup \{r_\beta \mid \beta \le \alpha\}$ is a set of 1-1 mutually disjoint $(n+1)$-tuples that avoids the increasing type. 

\begin{claim}
$r_\alpha$ is 1-1.
\end{claim}
\begin{proof}
By construction, $r_\alpha \upharpoonright n \not\in A_\alpha$. In particular, we have that $r_{\alpha, i} \neq f_0(\langle r_{\alpha, j}\rangle^{n-1}) = r_{\alpha, j}$ for all $j < i < n$. Moreover, by (2) of \cref{claim:thm:1}, we know that $r_{\alpha, n} \neq  f_0(\langle r_{\alpha, i}\rangle^{n-1}) = r_{\alpha, i}$ for all $i < n$. Thus, $r_\alpha$ is 1-1.
\end{proof}
\begin{claim}\label{claim.2}
For all $r \in P_\alpha$, $r$ and $r_\alpha$ are disjoint.
\end{claim}
\begin{proof}
Fix some $r \in P_\alpha$. The argument is analogous to the one of the previous claim. Since $r_\alpha \upharpoonright n \not\in A_\alpha$, we know that $f_0(\langle r_i\rangle^{n-1}) = r_i \not\in \mathrm{ran}(r_\alpha \upharpoonright n)$ for every $i \le n$. In other words, we know that $\mathrm{ran}(r_\alpha \upharpoonright n) \cap \mathrm{ran}(r) = \emptyset$. 
Moreover, by (2) of \cref{claim:thm:1}, we can also conclude that $r_{\alpha, n} \not\in \mathrm{ran}(r)$. Thus, $r$ and $r_\alpha$ are disjoint.
\end{proof}

\begin{claim}
For all $r \in P_\alpha$, $r$ and $r_\alpha$ do not realize the increasing type.
\end{claim}
\begin{proof}
Fix $r \in P_\alpha$. If $r \upharpoonright n$ and $r_\alpha \upharpoonright n$ do not realize the increasing type, then a fortiori $r$ and $r_\alpha$ do not. So we can assume that $r \upharpoonright n$ and $r_\alpha \upharpoonright n$ realize the increasing type. For the sake of simplicity, let us assume $t(r \upharpoonright n, r_\alpha \upharpoonright n) = \langle 0 \rangle^{n}$ (the other case is analogous).

It follows that $u_\alpha(r_\alpha \upharpoonright n) \le r_n$. By (1) of \cref{claim:thm:1}, we know that $r_{\alpha, n} \le u_\alpha(r_\alpha \upharpoonright n)$. Thus, given that $r_{\alpha, n} \neq r_n$ by Claim~\ref{claim.2}, we conclude that $r_{\alpha, n} < r_n$. In particular, $t(r, r_\alpha) = \langle 0\rangle^n{}^\smallfrown \langle 1\rangle$, which clearly is different from the increasing type.
\end{proof}

This ends the inductive definition. The set $E$ is not $(2^{\aleph_0}, n+1)$-entangled, since $\{r_\alpha \mid \alpha < 2^{\aleph_0}\}$ is a collection of size $2^{\aleph_0}$ of 1-1 mutually disjoint $(n+1)$-tuples of $E$ that avoids the increasing type. } Now note that if we order $2^{\aleph_0} \times (n+1)$ with the lexicographic ordering, then the enumeration $\langle r_{\alpha, i} \mid \alpha < 2^{\aleph_0}, i \le n\rangle$ of $E$ satisfies the hypothesis of \cref{prop:criterionforentangledness} by construction, and hence $E$ is $(2^{\aleph_0}, n)$-entangled.
\end{proof}

\begin{remark}
Actually, a slight modification of the argument of \cref{thm:1s} yields the following: if $\mathrm{cov}(\mathcal{B}) = 2^{\aleph_0}$, then, for every $n > 1$ and for every cardinal $\kappa \le 2^{\aleph_0}$ with $\mathrm{cf}(\kappa) = \mathrm{cf}(2^{\aleph_0})$, there exists a $(\kappa, n)$-entangled set of reals which is not $(\kappa, n+1)$-entangled.
\end{remark}

\begin{question}
Is the hypothesis $\mathrm{cov}(\mathcal{B}) = 2^{\aleph_0}$ necessary in \cref{thm:1s}?
\end{question}

\begin{question}
Let $n\geq2$. Can there be an $n$-entangled set but no $(n+1)$-entangled sets?\footnote{A few weeks after our preprint was made public, a positive answer was given by Cruz \cite{chapital2025nentangledsetn1entangledsets}.}
\end{question}

\section{Entangledness is not a topological property}\label{sec:thm2}

This section is devoted to the proof of \cref{thm:2s} and \cref{cor:2}, from which \cref{thm:2} directly follows.

Given two distinct $x, y \in {}^\omega 2$, we let $\Delta(x,y)$ be the least $n\in\omega$ such that $x(n) \neq y(n)$. A subset $S\subseteq {}^\omega 2$ is \emph{monochromatic} if the parity of $\Delta(x,y)$ is the same for all distinct $x,y \in S$.

The cardinal invariant $\mathfrak{hm}$, introduced in \cite{MR1942314}, is defined as the least cardinality of a family of monochromatic sets that cover ${}^\omega 2$. A peculiar property of $\mathfrak{hm}$ is that it cannot fall far behind the continuum, as $\mathfrak{hm}^+ \ge 2^{\aleph_0}$---we refer the reader to \cite{MR2114964,MR4562119} for more on this cardinal invariant. 

Regarding the statement of the next theorem, note that proving the existence of an entangled set of reals with a nontrivial autohomeomorphism is somewhat trivial: you can simply take an entangled set of reals and add two isolated points to it; it is still entangled and has a nontrivial autohomeomorphism, the one that simply swaps the two isolated points. It becomes less trivial if we ask our entangled set to be sufficiently crowded. Recall that a set $A \subseteq \mathbb{R}$ is said to be \emph{$\kappa$-crowded}, for some infinite cardinal $\kappa$, if $|U \cap A| \ge \kappa$ for every open subset $U \subseteq \mathbb{R}$ with $U \cap A \neq \emptyset$.

\begin{theorem}\label{thm:2s}
Assume $\mathfrak{hm} = 2^{\aleph_0}$. Then, there exists a $2^{\aleph_0}$-crowded, $2^{\aleph_0}$-entangled set of reals with a nontrivial autohomeomorphism.
\end{theorem}

Given a set $A\subseteq X\times Y$, we denote by $\mathrm{Proj}_X[A]$ the canonical projection of $A$ onto $X$. Moreover, given some $x \in X$, we denote by $A_x$ the set $\{y \in Y \mid (x,y) \in A\}$. If $A \subseteq X^n$, then we simply write $\mathrm{Proj}_k(A)$, with $k \le n$, instead of $\mathrm{Proj}_{X^k}(A)$. 

The following lemma plays a role similar to that of Kuratowski's extension theorem in the proof of \cref{prop:criterionforentangledness}.

\begin{lemma}\label{lemma:extension}
Given a Polish space $X$ and a set $A\subseteq X \times {}^\omega 2$ such that $A$ is closed in $\mathrm{Proj}_X[A]\times {}^\omega 2$ and, for all $x \in X$, $A_x$ is monochromatic, then there exists a $G_\delta$ set $\tilde{A}$ such that $A \subseteq \tilde{A} \subseteq \mathrm{cl}(A)$ and, for all $x \in X$, $\tilde{A}_x$ is monochromatic.
\end{lemma}
\begin{proof}
Consider the set $D = \{x \in X \mid \mathrm{cl}(A)_x \text{ is monochromatic}\}$. We claim that $D$ is $G_\delta$. Indeed, if we let $S_0$ and $S_1$ be the sets of all $s \in {}^{<\omega}2$ of even length and odd length, respectively, then, for all $x \in X$, the following holds:
\begin{align*}
x \in D \iff& \forall s \in S_0 \ (\mathrm{cl}(A)_x \cap N_{s^\smallfrown 0} = \emptyset \text{ or } \mathrm{cl}(A)_x \cap N_{s^\smallfrown 1} = \emptyset) \text{ or }\\& \forall s \in S_1 \ (\mathrm{cl}(A)_x \cap N_{s^\smallfrown 0} = \emptyset \text{ or } \mathrm{cl}(A)_x \cap N_{s^\smallfrown 1} = \emptyset).
\end{align*}
Here, $N_s$ is the standard open set of the Cantor space induced by $s$---i.e. the set of all the infinite binary sequences extending $s$. Using the sequential compactness of the Cantor space, an easy argument shows that the set of $x$'s such that $\mathrm{cl}(A)_x \cap N_{s} = \emptyset$ is open for every $s$. Therefore we conclude that $D$ is $G_\delta$.

Now simply let $\tilde{A} = \mathrm{cl}(A) \cap (D \times {}^\omega 2)$. It is clearly $G_\delta$. Since we assumed $A$ to be closed in $\mathrm{Proj}_X[A]\times {}^\omega 2$, we clearly have $A_x = \mathrm{cl}(A)_x$ for all $x \in \mathrm{Proj}_X[A]$. Furthermore, since $A_x$ is assumed to be monochromatic for all $x$, we conclude that $\mathrm{Proj}_X[A] \subseteq D$. In particular, $A \subseteq \tilde{A}$. 
\end{proof}

\begin{proof}[Proof of \cref{thm:2s}]
We work with the Cantor space ${}^\omega 2$, which we consider endowed with the lexicographic ordering. Note that the order-topology of the lexicographic ordering  coincides with the standard topology of the Cantor space. 

Consider the map $s_2:{}^\omega 2 \rightarrow {}^\omega 2$ defined by:  for every $x \in {}^\omega 2$ and $n\in\omega$,
\[
s_2(x)(n) = \begin{cases}
x(n)+1 \bmod 2 &\text{ if } n \bmod 2  = 0,\\
x(n)&\text{ otherwise}.
\end{cases}
\]

Note that a set $A \subseteq {}^\omega 2$ is monochromatic if and only if $s_2 \upharpoonright A$ is monotone.

We want to define a  $2^{\aleph_0}$-crowded and $2^{\aleph_0}$-entangled set $E\subseteq {}^\omega 2$ which is closed under the map $s_2$. If we manage to do such a construction we are done, as the map $s_2 \upharpoonright E$ would be a nontrivial autohomeomorphism of $E$.

To ensure crowdedness, fix a surjective map $\psi: 2^{\aleph_0} \rightarrow \omega$ with $|\psi^{-1}(n)| = 2^{\aleph_0}$ for every $n\in\omega$ and fix also an enumeration $\langle I_n \mid n \in\omega\rangle$ of the basic open sets of ${}^\omega 2$. 
Finally, fix an enumeration $\langle H_\alpha^n \mid n > 0 \text{ and } \alpha < 2^{\aleph_0} \rangle$ of the $G_\delta$ sets $H\subseteq ({}^\omega 2)^{n+1}$ such that $H_x$ is monochromatic for all $x \in ({}^\omega 2)^{n}$. We can suppose for simplicity that $H_0^1 = \{(x,x) \mid x \in {}^\omega 2\} \cup \{(x, s_2(x)) \mid x \in {}^\omega 2\}$.

Now we define inductively a sequence $\langle x_\alpha \mid \alpha < 2^{\aleph_0}\rangle$ of elements of ${}^\omega 2$ as follows: suppose we have defined the $x_\beta$'s for all $\beta < \alpha$, then we pick $x_\alpha$ inside $I_{\psi(\alpha)}$ and outside the set 
\begin{equation}\label{eq:outset}
\bigcup\big\{(H_\gamma^n)_x \mid n > 0 \text{ and } \gamma < \alpha \text{ and } x \in \{x_\beta \mid \beta < \alpha\}^{n}\big\}.
\end{equation}
Note that the set in \eqref{eq:outset}, being the union of less than $2^{\aleph_0}$-many monochromatic sets, does not cover $I_{\psi(\alpha)}$, since we are assuming $\mathfrak{hm} = 2^{\aleph_0}$.

Let $E = \{x_\alpha \mid \alpha < 2^{\aleph_0}\} \cup \{s_2(x_\alpha) \mid \alpha < 2^{\aleph_0}\}$. Clearly, $E$ is $2^{\aleph_0}$-crowded. We are left to prove that $E$ is $2^{\aleph_0}$-entangled.

To each pair $(\alpha, i)$ in $2^{\aleph_0} \times 2$, we associate the real $r_{\alpha, i}$ which is $x_\alpha$ if $i = 0$ and $s_2(x_\alpha)$ if $i = 1$. Note that the map $(\alpha, i) \mapsto r_{\alpha, i}$ is a bijection between $2^{\aleph_0} \times 2$ and $E$.

\begin{claim}
$E$ is $2^{\aleph_0}$-entangled.
\end{claim}
\begin{proof}

We prove that $E$ is $(2^{\aleph_0}, n)$-entangled for every $n> 0$ by induction on $n$. The base case $n=1$ is trivial, hence we assume that $E$ is $(2^{\aleph_0}, n)$-entangled  and fix a set $F$ of 1-1, mutually disjoint $(n+1)$-tuples of elements of $E$  with $|F| = 2^{\aleph_0}$ and a $t \in {}^{n+1} 2$ towards showing that $t$ is realized by two elements of $F$.

Seeing $F$ as a set of $(n+1)$-tuples of elements of $2^{\aleph_0} \times 2$, we can also assume without loss of generality that the $(n+1)$-tuples of $F$ are increasing with respect to the lexicographic ordering on  $2^{\aleph_0} \times 2$. We want to find two elements of $F$ realizing $t$.

Suppose first that there exists a subset $F' \subseteq F$ of size $2^{\aleph_0}$ such that for all $r \in F'$ there is a $\beta$ and an $(n-1)$-tuple $s$ such that $r = s^\smallfrown (x_\beta, s_2(x_\beta))$. Fix one such $F'$. We divide the argument in two cases, depending on whether $n = 1$:

\begin{description}
\itemsep0.3em
\item[$n=1$] We claim that $\mathrm{Proj}_1[F']$ is not monochromatic. Indeed, suppose towards a contradiction that it is. Then, there exists an $\alpha$ such that $H_\alpha^1 = {}^\omega 2\times \mathrm{cl}(\mathrm{Proj}_1[F'])$. By construction, $x_\beta \not\in \mathrm{Proj}_1[F']$ for all $\beta > \alpha$, which contradicts $F'$ having cardinality $2^{\aleph_0}$. Hence, $\mathrm{Proj}_1[F']$ is not monochromatic, or, equivalently, $s_2 \upharpoonright \mathrm{Proj}_1[F'] = F'$ is not monotone. Hence, every $t \in {}^2 2$ is realized by two elements of $F'$, as we wanted to show.

\item[$n > 1$] Denote by $J$ the closure of $\mathrm{Proj}_{n}[F']$. Let $D$ be the set of all those $x \in \mathrm{Proj}_{n-1}[F']$ such that $J_x$ is monochromatic. Suppose towards a contradiction that $|D| = 2^{\aleph_0}$. By \cref{lemma:extension}, there exists an $\alpha$ such that $H_\alpha^{n-1} \supseteq J \cap (D\times {}^\omega 2)$. Now pick some $n$-tuple $s^\smallfrown x_\beta \in \mathrm{Proj}_n[F']$ with $\beta > \alpha$ and $s \in D$---it exists because we are assuming $|D| = 2^{\aleph_0}$. In particular, $x_\beta \in (H_\alpha^{n-1})_s$, since $H_\alpha^{n-1} \supseteq J \cap (D\times {}^\omega 2) \supseteq \mathrm{Proj}_n[F'] \cap (D\times {}^\omega 2)$. But as $\alpha < \beta$, it follows directly from our construction that $x_\beta \not\in (H_\alpha^{n-1})_s$, hence the contradiction. 

Thus, $|D| < 2^{\aleph_0}$. In particular, there exists a set $B\subseteq \mathrm{Proj}_{n-1}[F']$ of size $2^{\aleph_0}$ such that $J_x$ is not monochromatic for all $x \in B$. We can assume without loss of generality that $t$ is either equal to $u^\smallfrown \langle 0,0\rangle$  or to  $u^\smallfrown \langle 1,0\rangle$ for some binary $(n-1)$-tuple $u$. If $t = u^\smallfrown \langle 0,0\rangle$ (resp. $u^\smallfrown \langle 1,0\rangle$), then let $C \subseteq B$ be such that $|C| = 2^{\aleph_0}$ and, for some finite binary sequence $v$ of odd (resp. even) length, $J_x$ has nonempty intersection with both $N_{v^\smallfrown 0}$ and $N_{v^\smallfrown 1}$ for all $x \in C$. By the induction hypothesis, $E$ is $(2^{\aleph_0}, n-1)$-entangled. Hence, there are two distinct elements $w,z \in  C$ such that $t(w,z) = u$. Since $J_w$ and $J_z$ have nonempty intersection with both $N_{v^\smallfrown 0}$ and $N_{v^\smallfrown 1}$, we can find $x,y \in F'$, with $x \upharpoonright (n-1)$ and $y \upharpoonright (n-1)$ sufficiently close to $w$ and $z$ respectively, such that:
\begin{enumerate}
\itemsep0.3em
\item $t(x \upharpoonright (n-1), y \upharpoonright (n-1)) = u$;
\item if $t = u^\smallfrown \langle 0, 0\rangle$, then $x(n-1) \in N_{v^\smallfrown 0}$ and $y(n-1) \in N_{v^\smallfrown 1}$; otherwise (that is, if $t = u^\smallfrown \langle 1, 0\rangle$), $x(n-1) \in N_{v^\smallfrown 1}$ and $y(n-1) \in N_{v^\smallfrown 0}$. 
\end{enumerate}
By the way we chose $F'$, we conclude that $t(x, y) = t$.
\end{description}

Now suppose that there is no such $F'$. This case is dealt with as in the standard construction of an entangled set. It follows from our hypothesis that there are $2^{\aleph_0}$-many $r \in F$ for which there is a $\gamma$ and some $n$-tuple $s \in \{x_\beta, s_2(x_\beta) \mid \beta < \gamma\}^n$ such that $r$ is either $s^\smallfrown x_\gamma$ or $s^\smallfrown s_2(x_\gamma)$. We just consider the case in which there exists a subset $F'' \subseteq F$ of size $2^{\aleph_0}$ such that, for all $r \in F''$, $r = s^\smallfrown s_2(x_\gamma)$ for some $\gamma$ and $s \in \{x_\beta \mid \beta < \gamma\}^n$, as the other cases are analogous. Let \lnote*{}{us} consider $F''$ as a map from $({}^\omega 2)^n$ to ${}^\omega 2$, and assume towards a contradiction that it has less than $2^{\aleph_0}$-many discontinuity points. Then, by restricting it if necessary, we can suppose that $F''$ is continuous. By \cref{lemma:extension} (or directly by Kuratowski's extension theorem) there is an $\alpha$ such that $s_2^{-1} \circ F'' \subseteq H_\alpha^n$. Since $F''$ has cardinality $2^{\aleph_0}$, there is a $\gamma$ with $\alpha < \gamma$ and an $s \in \{x_\beta \mid \beta < \gamma\}^n$ such that $s^\smallfrown s_2(x_\gamma) \in F''$ and, a fortiori, $s^\smallfrown x_\gamma \in H^n_\alpha$. But by construction $x_\gamma \not\in (H^n_\alpha)_s$, hence the contradiction. Thus, $F''$ has $2^{\aleph_0}$-many discontinuity points. Since, by induction hypothesis, $E$ is $(2^{\aleph_0}, n)$-entangled, we conclude by \cref{lemma:finalarg} that $t$ is realized by two elements of $F''\subseteq F$.
\end{proof}
This finishes the proof.
\end{proof}

\cref{thm:2} follows directly from \cref{thm:2s} together with the following fact.
\begin{lemma}\label{cor:2}
Suppose that there exists a $\kappa$-crowded, $\kappa$-entangled set of reals with a nontrivial autohomeomorphism. Then, there are two homeomorphic \lnote*{}{sets} of reals $A, B$ such that $A$ is $\kappa$-entangled but $B$ is not $(\kappa, 2)$-entangled.
\end{lemma}
\begin{proof}
Fix a $\kappa$-crowded, $\kappa$-entangled $E\subseteq \mathbb{R}$ and a nontrivial autohomeomorphism $h: E\rightarrow E$. In particular there exist two nonempty disjoint open intervals $I, J$ such that $|E \cap I| = |E \cap J| = \kappa$ and $h[E \cap I] = E \cap J$. Fix a nonempty open interval $W$ disjoint from both $I$ and $J$, and fix an increasing homeomorphism $f: J \rightarrow W$. 

Let $A = E \cap(I \cup J) $ and $B = (E \cap J) \cup f[E \cap J]$. The set $A$, being a subset of $E$ of cardinality $\kappa$, is $\kappa$-entangled. The set $B$, on the other hand, is not $(\kappa, 2)$-entangled, as witnessed by the increasing map $f$. Consider the following map
\begin{align*}
\varphi: A &\longrightarrow B\\
		x &\longmapsto \begin{cases}
		h(x) &\text{ if } x \in I,\\
		f(x) &\text{ otherwise.}
		\end{cases}
\end{align*}
The map $\varphi$ is a homeomorphism by disjointness of $I$ and $J$, and so we are done.
\end{proof}

\begin{remark}
As for \cref{thm:1s}, a slight modification of the argument of \cref{thm:2s} yields the following: if $\mathfrak{hm}= 2^{\aleph_0}$, then, for every $n > 1$ and for every cardinal $\kappa \le 2^{\aleph_0}$ with $\mathrm{cf}(\kappa) = \mathrm{cf}(2^{\aleph_0})$, there exists a $\kappa$-crowded, $\kappa$-entangled set of reals with a nontrivial autohomeomorphism.
\end{remark}

\begin{question}
Is $\mathfrak{hm} = 2^{\aleph_0}$ necessary in \cref{thm:2s}?
\end{question}

\begin{question}
Does it follow from $\mathsf{CH}$ that there exists an entangled set that is homogeneous as a topological space?
\end{question}

\section{Definable entangled sets}\label{sec:thm3}

In this section we prove \cref{thm:3}. Let us first observe that \cref{thm:3} is, arguably, optimal, since no analytic set can be $2$-entangled. Indeed, $2$-entangled sets do not satisfy the perfect set property:

\begin{proposition}
No $2$-entangled set of reals contains an uncountable closed set.
\end{proposition}
\begin{proof}
It suffices to show that no closed set of reals is $2$-entangled. Fix a closed uncountable $C \subseteq \mathbb{R}$, and fix a countable $D\subset C$ which is (topologically) dense in $C$ and such that $D$ is unbounded and dense as a linear order. By Cantor's isomorphism theorem, there exists a decreasing map $f: D \rightarrow D$, and, by the closure of $C$, $f$ can be extended to a decreasing map from $C$ to $C$. Therefore, $C$ is not $2$-entangled.
\end{proof}
\begin{corollary}
No analytic set of reals is $2$-entangled.
\end{corollary}
\begin{proof}
Analytic sets satisfy the perfect set property---i.e. every uncountable analytic set contains an uncountable closed set.
\end{proof}

In what follows we employ the notation of \cite[\S 13]{MR2731169}. Any real $x \in {}^\omega \omega$ encodes the binary relation $R_x = \{\langle m,k \rangle \mid x(\llangle m,k\rrangle) = 0\}$ on $\omega$, where $\llangle \cdot, \cdot\rrangle : \omega\times\omega \rightarrow \omega$ is a recursive bijection, and consequently it also encodes the structure
\[
M_x = \langle \omega, R_x\rangle
\]
for the language of set theory. Also, each real $x \in {}^\omega \omega$ encodes the sequence of reals $\langle (x)_i \mid i\in\omega \rangle$ defined by $(x)_i(j) = x(\llangle i,j\rrangle)$ for each $i,j\in\omega$.

\begin{theorem}
Suppose that $ \mathbb{R}\subseteq \mathsf{L}$. Then there exists a $\Pi_1^1$ set of reals which is entangled.
\end{theorem}
\begin{proof}
Assume that $\mathbb{R}\subseteq\mathsf{L}$ and denote by $<_{\mathsf{L}}$ the canonical well-ordering of $\mathsf{L}$.
Consider the set $\mathcal{A}$ of all the countable admissible\footnote{An ordinal $\alpha$ is \emph{admissible} if $\mathsf{L}_\alpha$ is a model of $\mathsf{KP}$. We just need that $\mathsf{L}_\alpha$ is a model of a sufficiently large fragment of $\mathsf{ZFC}$.} ordinals $\alpha$ such that $\mathsf{L}_\alpha \vDash $ ``Every set is countable". Note that $\mathcal{A}$ is unbounded in $\omega_1^\mathsf{L} = \omega_1$. The idea is to get a $\Pi_1^1$ set of unique codes for the structures $\langle \mathsf{L}_\alpha, \in \rangle$ with $\alpha \in \mathcal{A}$ (c.f. \cite[Theorem 13.12]{MR2731169}).

Consider the set $E\subseteq {}^\omega \omega$ of all reals $x$ that satisfy the conjunction of the following formul\ae:\vspace{0.3em}

\begin{enumerate}
\itemsep0.3em
\item $M_{(x)_0}$ is well-founded and extensional and $M_{(x)_0} \cong \langle \mathsf{L}_{\alpha}, \in\rangle$ for some  $\alpha \in \mathcal{A}$,

\item For every $i\in\omega$, $M_{(x)_{i+1}}$ is well-founded and extensional and $M_{(x)_{i+1}} \cong \langle \mathsf{L}_{\delta_{i+1}}, \in\rangle$, with $\delta_{i+1}$ being the least limit $\delta < \omega_1$ such that $(x)_i \in \mathsf{L}_\delta$,
\item For every $i\in\omega$, for every $z \in {}^\omega \omega \cap \mathsf{L}_{\delta_{i+1}}$, if $\mathsf{L}_{\delta_{i+1}} \vDash ``z <_\mathsf{L} (x)_i"$, then $M_z \not\cong M_{(x)_i}$.\vspace{0.3em}
\end{enumerate}
By rewriting (1)-(3) in the language of second-order arithmetic, we find that the set $E$ is $\Pi_1^1$---again, see \cite[Theorem 13.12]{MR2731169}. For each $x \in E$, let $\alpha_x$ be such that $M_{(x)_0} \cong \langle \mathsf{L}_{\alpha_x}, \in \rangle$.
\begin{claim}\label{claim1}
For each $x \in E$, $x \not\in \mathsf{L}_{\alpha_x}$.
\end{claim}
\begin{proof}
Suppose towards a contradiction that $x \in \mathsf{L}_{\alpha_x}$. Since $M_{(x)_0}$ is well-founded and extensional, a fortiori $\mathsf{L}_{\alpha_x} \vDash ``M_{(x)_0}$ is well-founded and extensional". But then, since $\alpha_x$ is admissible, the transitive collapse of $M_{(x)_0}$ also belongs to $\mathsf{L}_{\alpha_x}$. Since the transitive collapse of $M_{(x)_0}$ is $\mathsf{L}_{\alpha_x}$, we have reached a contradiction.
\end{proof}

Note that for each $\alpha \in \mathcal{A}$, there exists exactly one $x \in E$ with $\alpha_x = \alpha$. Moreover, note that, for each $\beta > \alpha$ with $\beta \in \mathcal{A}$, the (unique) real $x \in E$ with $\alpha_x = \alpha$ can already be found in $\mathsf{L}_\beta$---this is \lnote*{}{where} we use that $\mathsf{L}_\beta \vDash $``Every set is countable", to ensure that $\mathsf{L}_\alpha$ is countable in $\mathsf{L}_\beta$.

\begin{claim}
$E$ is entangled.
\end{claim}
\begin{proof}
Since for every $\alpha \in \mathcal{A}$ there exists a unique $x \in E$ with $\alpha_x = \alpha$, the well-order on $\mathcal{A}$ naturally induces a well-order on $E$. More precisely, let $\langle \alpha_\nu \mid \nu < \omega_1 \rangle$ be the increasing enumeration of $\mathcal{A}$, and let $\langle x_\nu \mid \nu < \omega_1\rangle$ be the enumeration of $E$ such that $\alpha_{x_\nu} = \alpha_\nu$ for every $\nu < \omega_1$.

It suffices to show that for every continuous map $f$ from a $G_\delta$ subset of $\mathbb{R}^k$ into $\mathbb{R}$ (for some $k \in\omega$), there exists a $\nu < \omega_1$ such that, for all $\mu \ge \nu$,
\begin{equation}\label{def:eq:1}
f\big[\{x_\eta \mid \eta < \mu\}^k\big] \cap E \subseteq \{x_\eta \mid \eta < \mu\},
\end{equation}
as then it follows from \cref{prop:criterionforentangledness} that $E$ is entangled.

Pick a continuous map $f$ from a $G_\delta$ subset of $\mathbb{R}^k$ into $\mathbb{R}$ (for some $k \in\omega$). The map $f$ is coded by some real $\bar{f} \in \mathbb{R}$. Let $\nu < \omega_1$ be large enough so that $\bar{f} \in \mathsf{L}_{\alpha_\nu}$---here we are using the hypothesis $\mathbb{R} \subseteq \mathsf{\mathsf{L}}$. Now pick some $\rho > \mu \ge \nu$ and some $s \in \{x_\eta \mid \eta < \mu\}^k$. Since both $\bar{f}$ and $s$ belong to $\mathsf{L}_{\alpha_\rho}$, it follows that $f(s)$ also belong to $\mathsf{L}_{\alpha_\rho}$. But then, since by \cref{claim1} $x_\rho \not\in \mathsf{L}_{\alpha_\rho}$, it follows $f(s) \neq x_\rho$.  Thus, \eqref{def:eq:1} holds.
\end{proof}
This finishes the proof.
\end{proof}

\section{A $2$-entangled non-separable linear order}\label{sec:thm4}

In this section we prove \cref{thm:4}. As we already noted in the introduction, every $2$-entangled non-separable linear order is, in particular, a Suslin line. It is well-known that the existence of a Suslin line follows from $\diamondsuit$ (e.g. \cite{MR756630,Devlin:1984aa}). We show that from the same hypothesis we can construct a $2$-entangled Suslin line. Recall that the existence of a $2$-entangled Suslin line is already known to be consistent by a forcing argument due to Krueger \cite{MR4080091}.

Recall that a partial order $(T, \le)$ is a \emph{tree} (in the order-theoretic sense) if for every $x \in T$, the set of the predecessors of $x$---i.e. the set $\{y \in T \mid y \le x\}$---is well-ordered with respect to $\le$. Given a node $x \in T$, the \emph{height} of $x$ is the order-type of $\{y \in T \mid y < x\} $. For every tree $(T, \le)$ and ordinal $\alpha$, we let  $\mathrm{Lev}_\alpha(T, \le)$ be the set of all the nodes $x \in T$ of height $\alpha$.

\begin{theorem}
Suppose that $\diamondsuit$ holds. Then, there exists a $2$-entangled non-separable linear order.
\end{theorem}
\begin{proof}
The construction is similar to the classical construction of a Suslin line under $\diamondsuit$. We define a tree ordering $\unlhd$ on $\omega_1\times \omega$ with $\mathrm{Lev}_\alpha(\omega_1 \times \omega, \unlhd) = \{\alpha\}\times \omega$ for every $\alpha$. For every $x = (\alpha, k) \in \omega_1\times\omega$, we denote $\alpha$ (resp. $k$) by $\mathrm{ht}(x)$ (resp. $n(x)$). Moreover, given $x,y \in \omega_1\times \omega$, we write $x \top y$ (resp. $x \perp y$) when $x$ and $y$ are comparable (resp. incomparable) with respect to $\unlhd$.

We want to define the tree $(\omega_1\times \omega, \unlhd)$ such that the following lexicographic-like linearization of $\unlhd$ is $2$-entangled and non-separable: for all $x,y \in \omega_1\times \omega$,
\begin{multline*}
x \preceq y \text{ if either } x \unlhd y \text{ or } \big( x \perp y \text{ and } \\
n(\min\{z \unlhd x \mid z \not\trianglelefteq y\}) >
n(\min \{z \unlhd y \mid z \not\trianglelefteq x\})\big),
\end{multline*}

\smallskip
\noindent where the minimum is taken with respect to $\unlhd$. From now on, the minima and maxima between nodes of our tree are taken with respect to $\preceq$.

We now introduce a useful binary relation induced by $\preceq$: given $x,y \in \omega_1 \times \omega$, we let $x \prec^+ y$ if $x \preceq y$ and $x \perp y$. The following straightforward observations are often implicitly used afterward: if $x \prec^+ y$, then $z \prec^+ y$ for all $z \unrhd x$; moreover, if $y \preceq x$, then $y \preceq z$ for all $z \unrhd x$.

For the sake of clarity, we denote by $L_\alpha$ the set  $\{\alpha\}\times \omega$ for each $\alpha$. Fix a $\diamondsuit$-sequence $\langle A_\alpha \mid \alpha < \omega_1\rangle$. Fix also a bijection $\psi: \omega_1 \rightarrow (\omega_1 \times \omega)^2$.  We define $\unlhd$ inductively such that:\vspace{0.3em}
\begin{enumerate}
\itemsep0.3em
\item $\unlhd$ is a tree ordering on $\omega_1 \times \omega$ and $\mathrm{Lev}_\alpha(\omega_1 \times \omega, \unlhd) = L_\alpha$ for every $\alpha < \omega_1$,
\item for every $n$ and $\alpha < \omega_1$, there are infinitely many $m$'s such that $(\alpha, n) \unlhd (\alpha+1, m)$,
\item if $\beta < \alpha < \omega_1$ and $x \in L_\beta$, then there is a $y \in L_\alpha$ with $x \lhd y$,
\item Suppose that $\alpha$ is limit and $\psi``A_\alpha \subseteq (\alpha\times \omega)^2$ is a map, denoted for clarity by $f_\alpha$, which is maximal among the 1-1 and monotone maps $f$ from a subset of $\alpha \times \omega$ to $\alpha \times \omega$ such that $f(x) \neq x$ for all $x \in \mathrm{dom}(f)$. Then, for all $x,y \in (\alpha+1) \times \omega$, if either

\smallskip
\begin{enumerate}[label=(4\alph*)]
\item $\mathrm{ht}(x) = \mathrm{ht}(y) = \alpha$, or
\item $\mathrm{ht}(x) < \alpha$ and $\mathrm{ht}(y) = \alpha$ and $x \not\in \mathrm{dom}(f_\alpha)$, or
\item $\mathrm{ht}(x) = \alpha$ and $\mathrm{ht}(y) < \alpha$ and $y \not\in \mathrm{ran}(f_\alpha)$,
\end{enumerate}

\smallskip
\noindent then the map $f_\alpha \cup \{(x,y)\}$ is not monotone.\vspace{0.3em}
\end{enumerate}

In (4) and onward, monotonicity is always with respect to $\preceq$.
Suppose for the moment that we have already defined $\unlhd$ satisfying (1)-(4). We claim that the linear order $(\omega_1 \times \omega, \preceq)$ is $2$-entangled and non-separable.

\begin{claim}
$(\omega_1 \times \omega, \preceq)$ is non-separable.
\end{claim}
\begin{proof}
Fix a countable $D\subseteq \omega_1 \times \omega$. Then, there is an $\alpha < \omega_1$ such that $D\subseteq \alpha \times \omega$. By (2), we can pick $x \in L_\alpha$  and $y \in L_{\alpha+1}$  with $x \lhd y$. Then, the open interval $]x, y[$ in $(\omega_1 \times \omega, \preceq)$ is nonempty as, by (2) again, it contains infinitely many elements of $L_{\alpha+1}$. Moreover, it is straightforward that ${]x,y[} \subseteq (\omega_1\setminus \alpha) \times \omega$. Thus, no element of $D$ belongs to $]x,y[$. We conclude that $(\omega_1 \times \omega, \preceq)$ is non-separable.
\end{proof}

\begin{claim}
$(\omega_1 \times \omega, \preceq)$ is $2$-entangled.
\end{claim}
\begin{proof}
Fix a 1-1 monotone map $f: A \rightarrow \omega_1 \times \omega$, with $A \subseteq \omega_1 \times \omega$, such that $f(x) \neq x$ for all $x \in A$ towards showing that $A$ is countable.

Let $\tilde{f}$ be a maximal 1-1 and monotone extension of $f$ such that $\tilde{f}(x) \neq x$ for all $x \in \mathrm{dom}(\tilde{f})$.

There is a club $C\subseteq\omega_1$ such that, for all $\alpha \in C$, $\alpha \times \omega$ is closed under both $\tilde{f}$ and $\tilde{f}^{-1}$ and  $\tilde{f}  \upharpoonright \alpha \times \omega$ is maximal among the maps from a subset of $\alpha \times \omega$ to $\alpha\times \omega$ that satisfy the same relevant properties of $\tilde{f}$---that is, being 1-1, monotone and pointwise different from the identity. By the properties of the $\diamondsuit$-sequence there is an $\alpha \in C$ with $f_\alpha = \tilde{f} \upharpoonright \alpha \times \omega$. We claim that $\mathrm{dom}(\tilde{f}) \subseteq \alpha \times \omega$, from which the countability of $A$ directly follows.

Suppose towards a contradiction that there exists some $x \in \mathrm{dom}(\tilde{f}) \setminus (\alpha \times \omega)$. Fix one such $x$. Note that, since $\alpha \times \omega$ is closed under $\tilde{f}^{-1}$, we also have $\tilde{f}(x) \not\in \alpha \times \omega$.  Let $x \upharpoonright \alpha$ be the $\unlhd$-predecessor of $x$ of height $\alpha$, and define $\tilde{f}(x) \upharpoonright \alpha$ analogously for $\tilde{f}(x)$.  By (4), the map $(\tilde{f} \upharpoonright \alpha \times \omega) \cup \{(x \upharpoonright \alpha, \tilde{f}(x) \upharpoonright \alpha)\}$ is not monotone.

It is straightforward that for every $z,w \in \omega_1\times \omega$ with $\mathrm{ht}(z) > \mathrm{ht}(w)$, $z \prec w$ if and only if $z \prec^+ w$. Therefore, it follows directly from the useful observation made after the definition of $\prec^+$ and from $(\tilde{f} \upharpoonright \alpha \times \omega) \cup \{(x \upharpoonright \alpha, \tilde{f}(x) \upharpoonright \alpha)\}$ not being monotone that also $(\tilde{f} \upharpoonright \alpha \times \omega) \cup \{(x, \tilde{f}(x))\}$ is not monotone, which is a contradiction.
\end{proof}

We define $\unlhd$ by induction on its levels. Suppose that we have defined $\unlhd$ on $\alpha \times \omega$ so that (1)-(4) holds below $\alpha$, towards extending $\unlhd$ to $(\alpha+1) \times \omega$.

We deal with the only nontrivial case: $\alpha$ is limit, and the hypotheses of (4) are satisfied. Fix an increasing sequence $(\xi_m)_{m\in\omega}$ of countable ordinals such that $\sup_m \xi_m = \alpha$ and fix also an enumeration $(b_n)_{n\in\omega}$ of $\alpha \times \omega$.

We define a family $\langle z^n_m \mid n,m\in\omega\rangle$ of elements of $\alpha \times \omega$ such that $\mathrm{ht}(z^n_m) \ge \xi_m$ and $z^n_{m} \unlhd z^n_{m+1}$ for every $n,m\in\omega$. Once we have defined this family, we extend $\unlhd$ as follows: for every $x \in \alpha \times \omega$ and $n\in\omega$, we let $x \lhd (\alpha, n)$ if $x \unlhd z^n_m$ for some $m$.

We describe how to define the $z^{n}_m$s and only then we show that this definition works---that is, the resulting extension of $\unlhd$ on $(\alpha+1) \times \omega$ still satisfies (1)-(4) below $\alpha+1$.

Fix an $n\in\omega$ and suppose that we have defined $z^k_m$ for all $k < n$ and all $m$, towards defining $z^{n}_m$ for all $m$. Fix an enumeration $(w_m, i_m)_{m\in\omega}$ of the set $((\alpha \times \omega) \cup (\{\alpha\} \times n))\times 2$.

First let us define $z_0^n$. Suppose first that there is no $z \in \mathrm{dom}(f_\alpha) \cup \mathrm{ran}(f_\alpha)$ such that  $b_n \lhd z$: in this case let $z_0^n$ be such that $b_n \lhd z_0^n$ and $\mathrm{ht}(z_0^n) = \max(\mathrm{ht}(b_n), \xi_0)+2$---we can find such $z_0^n$ by (3). 

Now suppose that there is a $z \in \mathrm{dom}(f_\alpha) \cup \mathrm{ran}(f_\alpha)$ such that  $b_n \lhd z$. Fix one such $z$. Let $z^\star = f_\alpha(z)$ if $z \in \mathrm{dom}(f_\alpha)$ and $z^\star = f_{\alpha}^{-1}(z)$ otherwise.  We need to consider the following cases:


\begin{enumerate}[label={(\roman*)}]
\itemsep0.3em
\item $z \top z^\star$ and $f_\alpha$ is increasing:
\newline by the way we chose $z$ and by case assumption, $z,z^\star$ and $b_n$ are all $\unlhd$-compatible and, moreover, $\max(\min(z,z^\star), b_n) \lhd \max(z,z^\star)$. Let $\delta = \mathrm{ht}(\max(\min(z, z^\star), b_n))+1$, and let $i \in \omega$ be such that $(\delta, i) \unlhd \max(z, z^\star)$. By (2), we can pick $j > i$ such that $\max(\min(z, z^\star), b_n) \lhd (\delta, j)$. By (3) we can pick $z_0^n$ such that $(\delta, j) \unlhd z_0^n$ and $\mathrm{ht}(z_0^n) \ge \xi_0$. In particular, we have $\min(z,z^\star) \prec z_0^n \prec^+ \max(z, z^\star)$. 

\item $z \top z^\star$ and $f_\alpha$ is decreasing:
\newline by (3) we can pick $z_0^n$ such that $\max (z,z^\star) \unlhd z_0^n$ and $\mathrm{ht}(z_0^n) \ge \xi_0$. In particular, we have $z,z^\star \preceq z_0^n$.

\item $z \prec^+ z^\star$ and $b_n \not\trianglelefteq z^\star$ and $f_\alpha$ is increasing:
\newline by (3) we can pick $z_0^n$ such that $z \unlhd z_0^n$ and $\mathrm{ht}(z_0^n) \ge \xi_0$. In particular, we have $z \preceq z_0^n \prec^+ z^\star$.

\item $z^\star \prec^+ z$ and $b_n \not\trianglelefteq z^\star$ and $f_\alpha$ is increasing:
\newline let $i \in \omega$ be such that $(\mathrm{ht}(b_n)+1, i) \unlhd z$. By (2), we can fix some $j > i$ such that $b_n \lhd (\mathrm{ht}(b_n)+1, j)$. By (3), we can pick $z_0^n$ such that $(\mathrm{ht}(b_n)+1, j) \unlhd z_0^n$ and $\mathrm{ht}(z_0^n) \ge \xi_0$. In particular, we have $z^\star \prec z_0^n \prec^+ z$.

\item $z \perp z^\star$ and $b_n \unlhd z, z^\star$ and $f_\alpha$ is increasing:
\newline by (3), we can pick $z_0^n$ such that $\min (z,z^\star) \unlhd z_0^n$ and $\mathrm{ht}(z_0^n) \ge \xi_0$. In particular, we have $\min (z,z^\star) \preceq z_0^n \prec^+ \max (z,z^\star)$.

\item $z \prec^+ z^\star$ and $b_n \not\trianglelefteq z^\star$ and $f_\alpha$ is decreasing:
\newline pick $z_0^n$ as in (iv). In particular, we have $z_0^n \prec^+ z,z^\star$.

\item $z^\star \prec^+ z$ and $b_n \not\trianglelefteq z^\star$ and $f_\alpha$ is decreasing:
\newline pick $z_0^n$ as in (iii). In particular, we have $z,z^\star \preceq z_0^n$.

\item $z \perp z^\star$ and $b_n \unlhd z, z^\star$ and $f_\alpha$ is decreasing:
\newline by (3), we can pick $z_0^n$ such that $\max (z,z^\star) \unlhd z_0^n$ and $\mathrm{ht}(z_0^n) \ge \xi_0$. In particular, we have $z,z^\star \preceq z_0^n$.
\end{enumerate}

Suppose that we have defined $z_m^n$ towards defining $z_{m+1}^n$. If either $i_m = 0$ and $w_m \in \mathrm{dom}(f_\alpha)$, or $i_m = 1$ and $w_m \in \mathrm{ran}(f_\alpha)$, then simply let $z_{m+1}^n$ be such that $z_{m}^n \unlhd z_{m+1}^n$ and $\mathrm{ht}(z_{m+1}^n) \ge \xi_{m+1}$---we can find such a $z_{m+1}^n$ by (3). 

Now assume that either $i_m = 0$ and $w_m \not\in \mathrm{dom}(f_\alpha)$, or $i_m = 1$ and $w_m \not\in \mathrm{ran}(f_\alpha)$. If there is no $z$ with $z_m^n \lhd z$ such that $z \in \mathrm{ran}(f_\alpha)$ if $i_m = 0$ and $z \in \mathrm{dom}(f_\alpha)$ if $i_m = 1$, then let $z_{m+1}^n$ be such that $z_m^n \lhd z_{m+1}^n$ and $\mathrm{ht}(z_{m+1}^n) \ge \xi_{m+1}$ and $z_{m+1}^n \neq w_m$. Otherwise, fix one such $z$ and let $z^\star$ denote $f^{-1}_\alpha(z)$ if $i_m  = 0$, and $f_\alpha(z)$ if $i_m = 1$. Now we need to consider the following cases:
\begin{enumerate}[label={(\alph*)}]
\itemsep0.3em
\item $f_\alpha$ is increasing and $z^\star \prec w_m$:
\newline let $i \in \omega$ be such that $(\mathrm{ht}(z_m^n)+1, i) \unlhd z$. By (2), there we can fix some $j > i$ such that $z_m^n \lhd (\mathrm{ht}(z_m^n)+1, j)$. By (3), we can pick $z_{m+1}^n$ such that $(\mathrm{ht}(z_m^n)+1, j) \unlhd z_{m+1}^n$ and $\mathrm{ht}(z_{m+1}^n) \ge \xi_{m+1}$. In particular, we have $z_{m+1}^n \prec^+ z$

\item $f_\alpha$ is decreasing and $w_m \prec z^\star$:
\newline as in (a).

\item $f_\alpha$ is increasing and $w_m \prec z^\star$:
\newline by (3), we can pick $z_{m+1}^n$ such that $z \unlhd z_{m+1}^n$ and $\mathrm{ht}(z_{m+1}^n) \ge \xi_{m+1}$.

\item $f_\alpha$ is decreasing and $z^\star \prec w_m$:
\newline as in (c).
\end{enumerate}

This ends the inductive definition of the $z_m^n$s. We are left to prove that (1)-(4) are still satisfied below $\alpha+1$. 

Since we chose the $z_m^n$s such that $\sup_m \mathrm{ht}(z_m^n) = \sup_m \xi_m = \alpha$ for every $n$, we conclude that $\mathrm{Lev}_\alpha((\alpha+1)\times \omega, \unlhd) = L_\alpha$---that is, (1) is satisfied. 

Condition (2) trivially holds since it holds below $\alpha$ by the induction hypothesis, and $\alpha$ is limit.

Moreover, since $b_n \unlhd z_0^n \lhd (\alpha, n)$ for every $n$ by choice of $z_0^n$, (3) also holds---that is, every element of $\alpha\times \omega$ has an $\unlhd$-successor in $L_\alpha$.

Finally, we need to check that (4) is satisfied. We prove it by induction in the following sense: we suppose that for all $x,y \in (\alpha \times \omega) \cup (\{\alpha\} \times n)$ satisfying one among (4a)-(4c), the map $f_\alpha \cup \{(x,y)\}$ is not monotone, towards showing that the same can be said for all $x,y$ taken in $(\alpha \times \omega) \cup (\{\alpha\} \times (n+1))$.

We will check only some representative cases, as the other are analogous. From now on we assume that $f_\alpha$ is increasing.

Let us show that $f_\alpha \cup \{((\alpha, n), (\alpha, n))\}$ is not increasing.  Suppose that there is a $z \in \mathrm{dom}(f_\alpha) \cup \mathrm{ran}(f_\alpha)$ such that  $b_n \lhd z$. Fix the $z$ satisfying this property we chose when defining $z_0^n$---we assume for simplicity $z \in \mathrm{dom}(f_\alpha)$, as the case $z \in \mathrm{ran}(f_\alpha)$ is analogous. Then, looking at the possible cases (i.e. (i), and (iii)-(v)), either $z_0^n \prec^+ z$ and $f_\alpha(z) \preceq z_0^n$, or $z_0^n \prec^+ f_\alpha(z)$ and $z \preceq z_0^n$. In particular, since $z_0^n \lhd (\alpha, n)$, either $(\alpha, n) \prec z$ and $f_\alpha(z) \prec (\alpha, n)$, or $z \prec (\alpha, n)$ and $(\alpha, n) \prec f_\alpha(z)$. In other words, the pairs $((\alpha, n), (\alpha, n))$ and $(z, f_\alpha(z))$ realize the type $\langle 1,0\rangle$. 

Now suppose instead that there is no $z \in \mathrm{dom}(f_\alpha) \cup \mathrm{ran}(f_\alpha)$ such that $b_n \lhd z$. By definition of $z_0^n$ and case assumption, we have $\mathrm{ht}(z_0^n) \ge \mathrm{ht}(b_n)+2$. Therefore, we can fix some $t$ such that $b_n \lhd t \lhd z_0^n$. By case assumption, both $t$ and $z_0^n$ do not belong to $\mathrm{dom}(f_\alpha) \cup \mathrm{ran}(f_\alpha)$. By maximality of $f_\alpha$, we conclude that there is an $x \in \mathrm{dom}(f_\alpha)$ such that $(x, f_\alpha(x))$ and $(t, z_0^n)$ realize $\langle 1,0\rangle$.From the latter observation and our case assumption, it follows that either $t \prec^+ x$ and $f_\alpha(x) \preceq z_0^n$, or $x \preceq t$ and $z_0^n \prec^+ f_\alpha(x)$.
Since $t \lhd z_0^n \lhd (\alpha, n)$, we conclude that the pairs $((\alpha, n), (\alpha, n))$ and $(x, f_\alpha(x))$ realize the type $\langle 1,0\rangle$. 

We have shown that $f_\alpha \cup \{((\alpha, n), (\alpha, n))\}$ is not increasing. Now fix some $k$ with $k < n$ and let us show that $f_\alpha \cup \{(\alpha, k), (\alpha, n)\}$ is not increasing. Fix  an $m$ such that $w_m = (\alpha, k)$ and $i_m = 0$. Clearly $w_m = (\alpha, k) \not\in \mathrm{dom}(f_\alpha)$.

Suppose first that there is a $z$ with $z_m^n \lhd z$ such that $z \in \mathrm{ran}(f_\alpha)$. Fix the $z$ satisfying this property we chose when defining $z_{m+1}^n$. Looking at the cases (a) and (c) in the definition of $z_{m+1}^n$, it is easy to see that either $f_\alpha^{-1}(z) \prec (\alpha, k)$ and $z_{m+1}^n \prec^+ z$, or $(\alpha, k) \prec f_\alpha^{-1}(z)$ and $z \preceq z_{m+1}^n$. Hence, the pairs $((\alpha, k), (\alpha, n))$ and $(f_\alpha^{-1}(z), z)$ realize $\langle 1,0\rangle$. 

Now suppose instead that there is no $z$ with $z_m^n \lhd z$ such that $z \in \mathrm{ran}(f_\alpha)$. Clearly, this means that $z_{m+1}^n \not\in \mathrm{ran}(f_\alpha)$. By the induction hypothesis, there is some $x \in \mathrm{dom}(f_\alpha)$ such that the pairs $(x, f_\alpha(x))$ and $((\alpha, k), z_{m+1}^n)$ realize $\langle 1,0\rangle$. But now note that, again by case assumption, we must have either $z_{m+1}^n \prec^+ f_\alpha(x)$ or $f_\alpha(x) \prec z_{m+1}^n$. Therefore, the pairs $(x, f_\alpha(x))$ and $((\alpha, k), (\alpha, n))$ also realize $\langle 1,0\rangle$.

We have shown that $f_\alpha \cup \{((\alpha, k), (\alpha, n))\}$ is not increasing. Now fix some $y \in \alpha \times \omega$ with $y \not\in \mathrm{ran}(f_\alpha)$ towards showing that $f_\alpha \cup \{((\alpha, n), y)\}$ is not increasing. Fix an $m$ such that $w_m = y$ and $i_m = 1$. 

If there is a $z$ with $z_m^n \lhd z$ such that $z \in \mathrm{dom}(f_\alpha)$, the argument used to show that there is some $x \in \mathrm{dom}(f_\alpha)$ such that  $(x, f_\alpha(x))$ and $((\alpha, n), y)$ realize $\langle 1,0\rangle$ is the same as in the previous case. So, let us assume that there is no such $z$. In particular, we must have $z_{m+1}^n \not\in \mathrm{dom}(f_\alpha)$. Moreover, by the way we chose $z_{m+1}^n$, we also have $z_{m+1}^n \neq w_m = y$. By maximality of $f_\alpha$, there is some $x \in \mathrm{dom}(f_\alpha)$ such that the pairs $(x, f_\alpha(x))$ and $(z_{m+1}^n, y)$ realize $\langle 1,0\rangle$. But now note that, again by case assumption, we must have either $z_{m+1}^n \prec^+ x$ or $x \prec z_{m+1}^n$. Therefore, the pairs $(x, f_\alpha(x))$ and $((\alpha, n), y)$ also realize $\langle 1,0\rangle$.

The other cases are analogous.
\end{proof}

\begin{question}
Can there be a 2-entangled linear order but no separable 2-entangled linear order?\footnote{A positive answer was given by Mart\'{i}nez-Ranero and Polymeris \cite{martinezranero2025entangledsuslinlinesoga}.}
\end{question}

\printbibliography

\end{document}